\documentclass[12pt,a4paper]{article}

\usepackage{amsmath,amsthm,epic, eepic}

\def\Int{\operatorname{Int}}
\def\diam{\operatorname{diam}}

\def\wt{\operatorname{wt}}
\def\lg{\operatorname{lg}}
\def\bord{\partial}

\let\bydef\emph

\def\var{ma\-ni\-fold}
\def\svar{sub\-ma\-ni\-fold}

\def\irr{ir\-re\-du\-ci\-ble}

\def\comp{comp\-res\-si\-ble}
\def\incomp{in\-comp\-res\-si\-ble}

\def\triang{tri\-an\-gu\-la\-tion}

\newtheorem{theo}{Theorem}[section]

\newtheorem{prop}[theo]{Proposition}
\newtheorem{lem}[theo]{Lemma}
\newtheorem{slem}[theo]{Subemma}

\theoremstyle{definition}
\newtheorem*{defi}{Definition}
\newtheorem*{claim}{Claim}

\newenvironment{rem}{\bigskip\emph{Remark.}}{}
\newenvironment{rems}{\bigskip\emph{Remarks.}\begin{itemize}}{\end{itemize}}

\def\Rr{\mathbf{R}}

\def\Nn{\mathbf{N}}

\def\calT{\mathcal{T}}

\def\calC{\mathcal{C}}

\begin{document}
\title{A JSJ splitting for triangulated open 3-manifolds}
\author{Sylvain Maillot}

\maketitle

\begin{abstract}
We give a sufficient condition for an open $3$-manifold to
admit a decomposition along properly embedded open
annuli and tori, generalizing the toric splitting of
Jaco-Shalen and Johannson.
\end{abstract}

\section{Introduction}

This paper is a continuation of the program started in~\cite{maillot:spherical}, whose goal
is to find geometric conditions under which an open $3$-manifold admits a canonical
decomposition. In this paper we are concerned with Jaco-Shalen-Johannson (JSJ)
splittings.

Let us first recall some standard terminology from $3$-manifold theory. Throughout the
paper we work in the PL category. A $3$-manifold is \bydef{irreducible} if every embedded
$2$-sphere bounds a $3$-ball.  If $M$ is an orientable $3$-manifold and $F\subset M$ an embedded orientable surface not homeomorphic to $S^2$, then $F$ is \bydef{incompressible} if the homomorphism
$\pi_1F\to\pi_1M$ induced by inclusion is injective. A $3$-manifold $M$ is \bydef{atoroidal}
if all incompressible tori in $M$ are boundary-parallel. It is \bydef{Seifert fibered}, or
a \bydef{Seifert manifold}, if it fibers over a $2$-dimensional orbifold.

Let $M$ be an orientable, irreducible $3$-manifold without boundary. When $M$ is
compact, Jaco-Shalen~\cite{js:seifert} and Johannson~\cite{joh:hom} found a canonical family of pairwise disjoint, embedded, \incomp\ $2$-tori which split $M$ into submanifolds that are either Seifert
fibered or atoroidal. This family, called the \bydef{JSJ-splitting} of the $3$-manifold $M$, has several additional properties; for instance, any \incomp\ torus
embedded in $M$ can be homotoped into a Seifert piece.

A simple way to construct the JSJ-splitting of a closed manifold was given in~\cite{ns:canonical}: let
$F_1,F_2\subset M$ be two embedded surfaces in $M$. We say that
$F_1$ is \bydef{homotopically disjoint} from $F_2$ if there exists
a surface $F'_1$ which is homotopic to $F_1$ and disjoint from $F_2$.  An \incomp\ torus
is called \bydef{canonical} if it is homotopically disjoint from any other \incomp\ torus.
Then when $M$ is closed, its JSJ-splitting can be constructed by simply taking a disjoint
collection of representatives of each homotopy class of canonical tori.

We want to extend this theory to noncompact manifolds. There are of course obvious
sufficient conditions under which the theory goes through, for instance if $M$ is homeomorphic
to the interior of a compact \var. Thus we are typically interested in open $3$-manifolds
with infinite topological complexity. Consider the following class of examples: let $X$ be any
orientable, \irr, atoroidal $3$-\var\ whose boundary is an incompressible open annulus,
different from $S^1\times\Rr\times [0,\infty)$. Let $Y$ be the product of $S^1$ with an
orientable surface of infinite genus and boundary a line. Then by gluing $X$ and $Y$
together along their boundary annuli, one obtains a $3$-\var\ $M$. Such a \var\ certainly has
infinite topological complexity in any reasonable sense of the word (e.g.~its fundamental
group is infinitely generated). Yet it has an obvious splitting along a properly
embedded \incomp\ annulus into an atoroidal piece $X$ and a Seifert piece $Y$.

If in addition, $X$ contains no properly embedded compact annuli other than those
parallel to the boundary, then the Seifert submanifold $Y$ will be `maximal' in the sense that
every incompressible torus can be homotoped into $Y$. Thus it seems to the author
that this splitting qualifies as a JSJ-splitting. Note that such a manifold $M$ contains no
canonical tori; by contrast, it contains plenty of noncanonical tori, which in a sense `fill up' the
Seifert piece of the JSJ-splitting. This motivates the following definition:

\begin{defi}
A \bydef{JSJ-splitting} of an open, orientable, \irr\ $3$-manifold $M$ is a locally finite collection $\calC$ of disjoint properly embedded \incomp\ tori and open annuli satisfying the following conditions:
\begin{enumerate}
\item Each component of $M$ split along $\calC$ is Seifert fibered or atoroidal.
\item Each torus of $\calC$ is canonical.
\item Every canonical torus is homotopic to a torus of $\calC$.
\item Every incompressible torus is homotopic into a Seifert piece of $M$ split along $\calC$.
\end{enumerate}
\end{defi}

Although it may seem plausible at first sight that every open $3$-manifold has a JSJ-splitting,
this is not true. Indeed, in~\cite{maillot:ex} there is a construction of a $3$-manifold $M_3$
containing an infinite collection of pairwise non-homotopic canonical tori that intersect
some fixed compact set essentially; therefore, any collection of tori containing a representative
of each homotopy class of canonical tori fails to be locally finite, however the representatives
are chosen.

Our main result is a sufficient condition for an open $3$-\var\ $M$ to admit a
JSJ-splitting. Before stating it, we give a few definitions and conventions.
We shall assume all $3$-\var s in this paper to be orientable.

Let $M$ be a $3$-\var.
If $F_1$ and $F_2$ are not homotopically disjoint, we will say that they \bydef{intersect essentially}.
We say that $F_1,F_2$ \bydef{intersect minimally} if they are in general position, and the number of components of $F_1\cap F_2$ cannot be reduced by homotoping one of the surfaces.

Note that by a theorem of Waldhausen, homotopy implies isotopy for compact surfaces in a Haken
\var\ $M$. This result is still true if $M$ is open and \irr\ because any homotopy takes place
in a compact subset of $M$, which can be embedded into a Haken \svar\ of $M$. Hence we
shall sometimes use the words ``homotopic'' and ``isotopic'' interchangeably.

Let $T$ be an \incomp\ torus in $M$. Assume that $T$ is not canonical, and $M$ is open.
We shall see (Proposition~\ref{special} below) that
there is a unique free isotopy class $\xi(T)$ of simple closed curves  in $T$ such that
if $T'$ is an \incomp\ torus which is not homotopically disjoint from $T$ and intersects $T'$
minimally, then $T\cap T'$ consists of curves belonging to $\xi(T)$. We say that a simple
closed curve in $T$ is \bydef{special} if it belongs to $\xi(T)$. The motivating example
is the following: let $\Sigma$ be an open Seifert fibered manifold. Then every incompressible
torus is vertical (i.e., isotopic to a union of fibers); when two such tori are homotoped so as
to intersect minimally, the intersection curves are isotopic to regular fibers of the Seifert
fibration, which is unique except in a few special cases.
Thus, in general, special curves should correspond to fibers of Seifert components of the JSJ-splitting
we are looking for.

Let $\calT$ be a \triang\ of $M$. We denote the $k$-skeleton of $\calT$ by
$\calT^{(k)}$. We say that $\calT$  has \bydef{bounded geometry} if there is a uniform upper bound
on the number of simplices containing a given vertex. The \bydef{size} of a subset $A\subset M$ is the minimal number of
$3$-simplices of $\calT$ needed to cover $A$. This can be used to give a rough notion
of distance (more precisely a \bydef{quasimetric} in the sense of~\cite{sm:seifert}) $d_\calT$ on $M$ as follows: given two points $x,y\in M$, we let $d_\calT(x,y)$ be the minimal size
of a path connecting $x$ to $y$, minus one. (Quasi)metric balls,
neighborhoods, diameter etc.~can be defined in the usual way.

If $\gamma\subset M$ is a curve which is in general position with respect to $\mathcal T$,
then we define the \bydef{length} of $\gamma$ as the cardinal of $\gamma\cap \calT^{(2)}$.
 If $F$ is a general position compact surface in $M$, then the \bydef{weight} of $F$ is the cardinal of
 $F\cap \calT^{(1)}$.  We say that $F$ is \bydef{normal} if it misses $\calT^{(0)}$ and meets
transversely each 3-simplex $\sigma$ of $\calT$ in a finite collection of
disks that intersect each edge of $\sigma$ in at most one point. We say that $F$ is
a \bydef{least weight} surface if it has minimal weight among all surfaces isotopic to $F$.

\begin{theo}~\label{th:main}
Let $M$ be an open, orientable, \irr\ $3$-\var. Let $\calT$ be a \triang\ of bounded geometry
on $M$. Assume that the following hypotheses are fulfilled:
\begin{enumerate}
\item[\rm A)] There is a constant $C_1$ such that each isotopy class of canonical tori has a representative
of weight at most $C_1$.
\item[\rm B)] There is a constant $C_2$ such that for each least weight normal noncanonical \incomp\ torus $T$ and each point $x\in T$, there is a special curve $\gamma$ passing through $x$ and of length at
most $C_2$.
\item[\rm C)] There are constants $C_3,C_4$ such that if $T$ is a least weight normal noncanonical \incomp\ torus and $\sigma$ a $3$-simplex of $\calT$ such that $T\cap\sigma$ is disconnected, then
there exists a normal torus $T'$ of weight at most $C_3$, which is not homotopically disjoint
from $T$, and such that $T\cap T'$ contains a point whose distance from $\sigma$ is at most $C_4$.
\item[\rm D)] If $T$ and $T'$ are two disjoint least weight normal noncanonical tori, $\sigma$ is
a $3$-simplex such that $T\cap\sigma$ (resp.~$T'\cap\sigma$) consists of a single
disk $D$ (resp.~$D'$), then one can find an annulus $A$ connecting a special curve on $T$ to
a special curve on $T'$, and such that $\Int A$ does not intersect $T,T'$, and that there
is an essential arc $\alpha\subset A$ such that $\alpha\subset\sigma$ and $\alpha$
connects $D$ to $D'$.
\end{enumerate}
Then $M$ admits a JSJ-splitting.
\end{theo}

\begin{rems}

\item Condition A is an obvious way to rule out the phenomenon of Example $M_3$ in~\cite{maillot:ex}
where the pathology comes from canonical tori.
One might wonder whether it is sufficient. The answer is probably no. Example $\mathcal{O}_5$
of the same paper is an \emph{orbifold} with no canonical toric $2$-suborbifolds and yet no JSJ-splitting
(the definitions being extended to orbifolds in the obvious way.)
It seems highly likely that manifolds with the same property exist.

\item It is tempting to replace conditions B--D with simpler conditions, e.g.~requiring that
A be true even for noncanonical tori. However, such a hypothesis would be unreasonably
strong in the sense that it would rule out examples as well as counterexamples.
This is best explained by analogy with surfaces. Imagine that $M$ is a $3$-manifold
with a JSJ-splitting, $\Sigma$ a Seifert piece, and $g$ is a complete Riemannian metric whose restriction to $\Sigma$ is obtained by lifting a metric on the base orbifold $F$. Then
incompressible tori in $\Sigma$ correspond to curves in $F$, with area corresponding
to length. If one thinks that $M$ is triangulated in such a way that the simplices are
very small and of roughly uniform size, then the weight of the tori in $\Sigma$ corresponds to the
length of the curves in $F$, and  two tori that intersect a common $3$-simplex
correspond to two curves which are very close.

Now to assume that every simple closed curve in $F$ is homotopic
to a curve of uniformly bounded length would be too restrictive: it would force $F$
to be topologically finite. By contrast, it is reasonable to assume that for each long
geodesic $\gamma$ that comes very close to itself, there is a short geodesic intersecting
essentially $\gamma$ near the region where this happens. Condition C is analogous
to this, and essentially says that there are `enough small tori' to generate the Seifert
pieces of the JSJ-splitting we are looking for. Likewise, conditions B and D respectively
mean that the fibers of the Seifert pieces can be represented uniformly by small
curves, and that different Seifert pieces are sufficiently far apart from one another.
\end{rems}

We now give an informal outline of the proof of Theorem~\ref{th:main}. It is somewhat oversimplified
since we ignore such technical complications as Klein bottles and nonseparating tori,
but should help the reader understand the main ideas. We use the Jaco-Rubinstein theory
of PL minimal surfaces, which is reviewed in section~\ref{sec:pl}.

If all \incomp\ tori in $M$ are canonical, one takes a least PL area representative of
each homotopy class. This gives a possibly infinite collection $\mathcal C$ of canonical
tori. Since members of $\mathcal C$ are canonical, they are homotopically disjoint.
Since they have least PL area in their respective homotopy classes, they are in fact
disjoint. By hypothesis~A, they have uniformly bounded weight. From this, it is not
difficult to show that they have uniformly bounded diameter. Moreover, the collection
$\mathcal C$ is locally finite by Haken's finiteness theorem.
Hence the difficult case is when there are infinitely many noncanonical tori. In this case, we
certainly do not expect PL least area tori to have uniformly bounded diameter. Hence they
can accumulate, and the point is to understand how they accumulate.

More precisely, we need to use
those tori to build Seifert \svar s. For this, there is a well-known
construction~(cf., e.g., \cite{scott:fake}): letting $T,T'$
be two noncanonical tori intersecting minimally, one takes a regular neighborhood  $\Sigma$ of
$T\cup T'$. This neighborhood has a fibration by circles such that $T\cap T'$ is a union
of fibers. Its boundary consists of tori, which may or may not be \incomp. The compressible
ones bound solid tori, to which the fibration on $\Sigma$ can be extended, possibly with
exceptional fibers.

By iterating this construction, one builds a increasing sequence of Seifert submanifolds of $M$
engulfing more and more noncanonical tori. The main difficulty is that while the union
of those submanifolds is still Seifert fibered, the closure of this union need not be. In fact,
it need even not be a manifold. This is where we use hypotheses B, C, and D.

The paper is organized as follows. In Section~\ref{sec:pl}, we review normal surfaces and PL minimal surfaces.
In Section~\ref{sec:seifert} we review Seifert manifolds and prove Proposition~\ref{special},
which justifies the definition of special curves. In Section~\ref{sec:graph} we deal with graph submanifolds and state a technical result, Theorem~\ref{th:graph}, from which Theorem~\ref{th:main}
is easily deduced. Section~\ref{sec:proof} contains the proof of Theorem~\ref{th:graph}.

\section{Normal surfaces and PL area}\label{sec:pl}

Recall from~\cite{sm:seifert} the definition of a \bydef{regular
Jaco-Rubinstein metric} on $(M,\calT)$: it is a Riemannian metric on
$\calT^{(2)}-\calT^{(0)}$ such that each $2$-simplex is sent
isometrically by barycentric coordinates to a fixed ideal triangle in
the hy\-per\-bo\-lic plane. The crucial property for
applications to noncompact manifolds is that for every number $n$,
there are finitely many subcomplexes of size $n$ up to isometry. 
Let $F\subset M$ be a compact, orientable, embedded surface in general position with
respect to $\calT$. Recall from the introduction that the \bydef{weight} $\wt(f)$ of $f$ is
the cardinal of $F\cap \calT^{(1)}$.
Its \bydef{length} $\lg(F)$ is the total length of all the arcs in the
boundaries of the disks in which $F$ intersects the 3-simplices of
$\mathcal T$. The \bydef{PL area} of $F$ is the pair 
$|F|=(\wt(F),\lg(F))\in \Nn\times \Rr_+$. We are interested in surfaces
having least PL area among surfaces in a particular class with 
respect to the lexicographic order.

We collect in the next proposition some existence results and properties of PL least area surfaces. We are mostly interested in the case of tori and Klein bottles. 
\begin{prop}\label{pl stuff}
\begin{enumerate}
\item Let $F\subset M$ be an \incomp\ normal surface that is not homotopic to a double cover of
an embedded nonorientable surface. Then there is a unique normal surface $F_0$
which is normally homotopic to $F$ and has least PL area among such surfaces.
\item Let $F\subset M$ be an \incomp\ surface. Then there is a normal surface $F_0$ which has
least weight among all surfaces homotopic to $F$. Furthermore, if $F$ is not homotopic to
a double cover of an embedded nonorientable surface, then there is a normal
surface $F_0$ which has least PL area among all surfaces homotopic to $F$.
\item Let $F,F'\subset M$ be \incomp\ surfaces that have minimal PL area in their respective
homotopy classes. If $F$ and $F'$ are homotopically disjoint, then they are disjoint or equal.
Otherwise, after a small perturbation they intersect minimally.
\end{enumerate}
\end{prop}

\begin{proof}
When $M$ is compact, (i) follows from~\cite[Theorem 2]{jr:min}, (ii) from~\cite[Theorems 3, 4, 6]{jr:min}; the first part of (iii) follows from~\cite[Theorem 7]{jr:min}, and the second part,
though not explicitly recorded in~\cite{jr:min}, can be
proved by copying the proof of the corresponding result of~\cite{fhs} for Riemannian
least area surfaces.

The proofs are easily extended to the noncompact case since the Jaco-Rubinstein metric
is regular, as noted in~\cite[Appendix A]{sm:seifert}.
\end{proof}

In the sequel, we shall call ``least PL area surface'' an embedded normal surface that has least
PL area in its homotopy class.

\begin{rem}
If an \incomp\ torus $T$ is homotopic to a double cover of a Klein bottle embedded in $M$, then
it bounds a \svar\ of $M$ homeomorphic to $K^2\tilde\times I$. This implies that $T$ is
canonical. Hence the second sentence of Proposition~\ref{pl stuff}(ii) always applies to
noncanonical tori. If $T$ is a canonical torus, however, there may be a least PL area
torus homotopic to $T$, or a least PL area Klein bottle $K$ such that $T$ is homotopic
to a double cover of $K$, or both.
\end{rem}

In Section~\ref{sec:proof} we shall need a notion of normality for
surfaces that are not properly embedded in $M$ (arising as compact
subsurfaces of noncompact normal surfaces.) We adopt the following definitions:
let $\gamma$ be a $1$-\svar\ of $M$. We say that it is \bydef{normal} if it is in general
position with respect to $\calT$, and for every $3$-simplex $\sigma$ of $\calT$, each
component of $\gamma\cap\sigma$ is an arc connecting two distinct faces of $\sigma$.
(This notion is akin to M.~Dunwoody's \emph{tracks}.) If $F$ is a normal surface in $M$,
and $F'$ is a subsurface (with boundary) of $F$, then $F'$ is \bydef{normal} if its boundary
is normal. The notion of \bydef{normal homotopy} extends to normal curves in the
obvious way.

The following lemma from~\cite{sm:seifert} provides a useful
inequality between the weight of a normal surface and the diameter of
its image with respect to the quasimetric $d_\calT$. The proof can be extracted from
Lemmas~2.2 and~A.1 of~\cite{sm:seifert}. (There the surface $F$ is supposed to
be properly embedded, but the extension is immediate.)

\begin{lem}\label{myineq}
Let $F$ be a compact, not necessarily properly embedded, normal surface.Then
$\diam(F)\le \wt(F)^2$.
\end{lem}

It has the following important consequence:
\begin{prop}\label{loc fin}
Let $\mathcal C$ be a collection of pairwise nonisotopic closed normal
surfaces. If members of $\mathcal C$ have uniformly bounded weight, then $\mathcal C$ is
locally finite.
\end{prop}

\begin{proof}
Let $K$ be a compact subset of $M$. Let $\mathcal C_K$ be the collection of members of
$\mathcal C$ that meet $K$. By Lemma~\ref{myineq}, there is a uniform bound on the
diameter of members of $\mathcal C_K$, which implies that they are all contained in some
finite complex $K'$. Since there are only finitely many normal surfaces of given weight in $K'$ up to
isotopy, Proposition~\ref{loc fin} follows.
\end{proof}

\section{Seifert submanifolds}\label{sec:seifert}

We will use classical results on Seifert fiber spaces. For proofs we refer
to~\cite{jaco:lect}~and~\cite{js:seifert}.

We call a Seifert \var\ \bydef{large} if its Seifert fibration is unique up to isotopy. Recall
that if $\Sigma$ is compact with nonempty \incomp\ boundary, then $\Sigma$ is
large unless $\Sigma$ is homeomorphic to the thickened torus $T^2\times I$ or
the twisted $I$-bundle over the Klein bottle $K^2\tilde\times I$.

Below is the proposition that justifies the definition of special curves on noncanonical tori.
\begin{prop}\label{special}
Let $T$ be an \incomp\ torus in $M$. Assume that $T$ is not canonical, and not Seifert fibered.
Then there is a unique free isotopy class $\xi(T)$ of simple closed curves  in $T$ such that
if $T'$ is an \incomp\ torus which is not homotopically disjoint from $T$ and intersects $T'$
minimally, then $T\cap T'$ consists of curves belonging to $\xi(T)$. 
\end{prop}

\begin{proof}
Let $T'$ be an \incomp\ torus which intersects $T$ essentially and minimally. Then all components
of $T\cap T'$ are essential on $T$. Since $T$ is a torus, they are all freely isotopic. Let
$\xi(T)$ be their free isotopy class. All we have to do is check that $\xi(T)$ does not depend
on the choice of $T'$.

Let $\Sigma'$ be a regular neighborhood of $T\cup T'$. Then $\Sigma$ carries an obvious
fibration such that $T\cap T'$ is a union of fibers. Since $T,T'$ are \incomp\ in $M$, the generic fiber of this Seifert fibration is not contractible in $M$. The components of $\Sigma'$ are tori
$T_1,\ldots,T_k$. Let $T_i$ be one of them. Since $T_i$ contains a fiber of the Seifert fibration,
it is not contained in a $3$-ball. Hence either $T_i$ is \incomp\ in $M$ or $T_i$ bounds a solid torus $V_i$. In the latter case, since $\Sigma'$ contains \incomp\ tori, $V_i$ cannot contain $\Sigma'$, which implies that its interior is disjoint from $\Sigma'$.

Let  $\Sigma$ be the manifold obtained from $\Sigma'$ by capping off the solid torus $V_i$ for
each \comp\ $T_i$. Since the generic fiber of $\Sigma'$ is noncontractible
in $M$, the Seifert fibration of $\Sigma'$ extends to $\Sigma$. Furthermore, $\bord \Sigma$ is \incomp\ in $M$. Since $M$ is open, $\bord\Sigma$ is not empty. Since $\Sigma$ contains two nonisotopic \incomp\ tori, $\Sigma$ must be large.

Let $T''$ be another \incomp\ torus meeting $T$ essentially and minimally. Let $\eta$ be a component
of $T\cap T''$. Our goal is to show that $\eta\in\xi(T)$. For this it is enough to prove that
$\eta$ is freely homotopic to the generic fiber of $\Sigma$.

Let $\Sigma''$ be the complement of a regular neighborhood of $T$ in $\Sigma$. Then
$\Sigma''$ inherits a Seifert fibration. Call $T_1,T_2$ the two components of $\bord\Sigma''$
parallel to $T$. Set $U:=T''\cap\bord\Sigma''$. By an isotopy of $T''$ fixing $\eta$, we may
assume that $U$ consists entirely of essential curves. Call $\eta_1,\eta_2$ the two
components of $U$ parallel to $\eta$ (with $\eta_i\subset T_i$).

Observe that $T\cap\Sigma''$ consists of finitely many annuli connecting the various components of $U$. Let $A_1,A_2$ be the annuli containing $\eta_1,\eta_2$ respectively. If at least one of these annuli
are vertical in $\Sigma''$, then we are done. Assume they are horizontal. Then $\Sigma''$ cannot be connected, for if it were, then it would have at least three boundary components, but such Seifert manifolds cannot contain horizontal annuli. Hence $\Sigma''$ has two components. Call $\Sigma_i$ the one that contains $T_i$, for $i=1,2$. At least one of them, say $\Sigma_1$, has more than one
boundary component. Hence the base orbifold of $\Sigma_1$ must be a nonsingular annulus,
which shows that $\Sigma_1\cong T^2\times I$. Now each component of $\Sigma_1\cap T''$ is
an \incomp\ annulus connecting $T_1$ to itself, so these annuli must be boundary-parallel. This
contradicts the minimality of $T''\cap T$.
\end{proof}

Note that in the course of the proof we have shown:
\begin{lem}\label{create Seifert}[cf.~\cite[Lemma 3.2]{scott:fake}
Let $T,T'$ be two \incomp\ tori which intersect essentially and minimally. Let
$U$ be a regular neighborhood of $T\cup T'$. Then $U$ is
contained in some large Seifert \svar\ $\Sigma$ whose boundary is \incomp\ in $M$, and
such that any special curve on $T$ is isotopic to a fiber of $\Sigma$.
\end{lem}

A similar argument shows:
\begin{lem}
Let $T$ be a noncanonical \incomp\ torus in $M$ contained in some large Seifert \svar\ $\Sigma$.
Then any special curve on $T$ is isotopic to a regular fiber of $\Sigma$.
\end{lem}

Next is a lemma that will enable us to enlarge a \svar\ $\Sigma$
of $M$ so as to engulf all \incomp\ tori that are homotopic into $\Sigma$. 
\begin{lem}\label{region X}
Let $\Sigma\subset M$ be a \svar\ bounded by PL minimal \incomp\ tori. Let $T$ be a component
of $\bord\Sigma$. Assume that $\Sigma$ is not a thickened torus. Then one of the following holds:
\begin{enumerate}
\item There is a \svar\ $X$ homeomorphic to $T^2\times [0,\infty)$ such that $X\cap\Sigma=T$, or
\item All least area tori isotopic to $T$ lie in $\Sigma$, or
\item There is an \incomp\ least PL area torus $T'$ isotopic to $T$ such that $T$ and
$T'$ cobound a thickened torus $X$ with $X\cap\Sigma-T$, and all least area tori isotopic
to $T$ lie in $\Sigma\cup X$, or
\item There is a  \svar\ $X$ homeomorphic to $K^2\tilde\times I$ such that
$X\cap \Sigma = T$.
\end{enumerate}
\end{lem}

\begin{proof}
By Proposition~\ref{pl stuff}(iii), least area \incomp\ tori are equal or disjoint.
Thus if neither of cases (ii), (iii)  and (iv) holds, then one can inductively construct an infinite sequence of
least area tori $T_n$ such that each $T_n$ cobounds with $T$ a thickened torus $X_n$, with the
property that $X_n\cap \Sigma=T$, and $\Sigma\cup X$ does not contain all PL least area
tori isotopic to $T$, and $X_n\subset X_{n+1}$. Since all $T_n$'s have the same weight,
a compact subset of $M$ can contain at most finitely of them, by Proposition~\ref{loc fin}. Hence the $X_n$ exhaust
a tame end of $M$ and conclusion (i) holds.
\end{proof}

\section{Graph submanifolds}\label{sec:graph}

For the purposes of this paper, a \bydef{graph \svar} of $M$ is a $3$-\svar\ $\Sigma\subset M$ that contains a locally finite collection of pairwise disjoint, least weight normal canonical tori $\mathcal G$
such that each component of $\Sigma$ split along $\mathcal G$ is Seifert fibered.
This notion is needed in order to address a technical difficulty caused by nonseparating tori.

We now formulate the main technical result of this article, from which
Theorem~\ref{th:main} will follow.

\begin{theo}~\label{th:graph}
Under the hypotheses {\rm A, B, C} and {\rm D}, there exists a \svar\ $\Sigma\subset M$
(possibly empty or equal to $M$) with the following properties~:
\begin{enumerate}
\item $\Sigma$ is a graph submanifold of $M$.
\item If $T$ is a noncanonical least PL area \incomp\ torus in $M$, then $T\subset\Sigma$.
\item If $F$ is a least PL area canonical torus or a least PL area Klein bottle, then
either $F\subset \Sigma$ or  $F\cap\Sigma=\emptyset$.
\end{enumerate}
\end{theo}

In Subsection~\ref{sub:one}, we show how to deduce Theorem~\ref{th:main} from Theorem~\ref{th:graph}. In Subsection~\ref{sub:two} we introduce a notion
of \emph{taut} graph submanifold and prove an important technical result on the
existence of such submanifolds.

\subsection{Deduction of Theorem~\ref{th:main} from Theorem~\ref{th:graph}}
\label{sub:one}

Let $\Sigma$ be a graph \svar\ satisfying the conclusion of  Theorem~\ref{th:graph}.
Let $\mathcal G$ be a  locally finite collection of pairwise disjoint, least weight normal canonical tori $\mathcal G$ such that each component of $\Sigma$ split along $\mathcal G$ is Seifert fibered.

For each isotopy class of canonical tori $\xi$, pick either a PL least area representative
or a PL least area Klein bottle that is double covered by a torus of $\xi$; this yields a
collection $\mathcal C$ of PL least area surfaces embedded in $M$. By Proposition~\ref{pl stuff} those
surfaces are disjoint. Moreover, we may assume that for every $T\in\mathcal G$, if $T$ is
not homotopic to a double cover of a Klein bottle, then $T$ is least PL area and $T\in \mathcal C$.

By hypothesis~A and Proposition~\ref{loc fin}, $\mathcal C$ is locally finite.
By part~(iii) of the conclusion of Theorem~\ref{th:graph}, $\mathcal C$ is the disjoint
union of four collections $\mathcal C_1,\mathcal C_2,\mathcal C_3,\mathcal C_4$, where $\mathcal C_1$ consists
of the collection of tori in $\mathcal C_1$ that are contained in $\Sigma$,
$\mathcal C_2$ consists of the collection of those that are disjoint from $\Sigma$,
and $\mathcal C_3$ (resp.~$\mathcal C_4$) consists of the Klein bottles that are
contained in $\Sigma$ (resp.~are disjoint from $\Sigma$).

For each Klein bottle $K$ in $\mathcal C_3$, perturb $K$ to
a normal torus $T$ that bounds a \svar\ $X$ homeomorphic to $K^2\tilde\times I$, such
that $X\subset \Sigma$. Since $\mathcal C$ is locally finite, this can be done in such a way
that  all tori in the resulting collection $\mathcal C'_3$ are disjoint from
one another and from members of $\mathcal C_1$.
As a consequence, each component of $\Sigma$ split along $\mathcal C_1\cup
\mathcal C'_3$ is Seifert fibered. 

Let $\Sigma'$ be the Seifert \svar\ of $M$ obtained by taking all components
of $\Sigma$ split along $\mathcal C_1\cup\mathcal C'_3$, and adding a ``small'' regular neighborhood of each
member of $\mathcal C_2\cup\mathcal C_4$.  (Here ``small'' means that these neighborhoods
should be disjoint from one another and from $\Sigma$. Again, this is possible because
$\mathcal C$ is locally finite.)  Each component of $\Sigma'$ is
either a component of $\Sigma$ split along $\mathcal C_1\cup\mathcal C'_3$ or homeomorphic
to $T^2\times I$ or $K^2\tilde\times I$, hence $\Sigma'$ is a Seifert \svar.

\begin{claim}
Every incompressible  torus is homotopic into $\Sigma'$.
\end{claim}

This is clear for canonical tori. Let $T$ be a noncanonical
torus. Then $T$ can be isotoped into $\Sigma$ by conclusion~(ii) of Theorem~\ref{th:graph}. Since members of $\mathcal C_1\cup C'_3$ are canonical, $T$ can in fact be isotoped
into $\Sigma'$. This proves the claim.

In particular, every component of $M\setminus \Sigma'$ is atoroidal.
This completes the proof of Theorem~\ref{th:main} assuming Theorem~\ref{th:graph}.

\subsection{Taut graph submanifolds}
\label{sub:two}

\begin{defi}
Let $\Sigma\subset M$ be a graph \svar. It is \bydef{taut} if it has \incomp\ boundary, no two
components of $\bord\Sigma$ are parallel, every boundary torus is PL least area,
and any PL least area \incomp\ torus which can be homotoped into $\Sigma$ is
already contained in $\Sigma$.
\end{defi}

\begin{prop}\label{tautening}
\begin{enumerate}
\item If $T,T'$ are noncanonical, PL least area \incomp\ tori which intersect essentially,
then there exists a taut graph \svar\ $\Sigma$ containing both $T$ and $T'$.
\item If $\Sigma$ is a taut graph \svar\ and $T$ is a noncanonical PL least area \incomp\ torus,
then there exists a taut graph \svar\ $\Sigma'$ containing both $\Sigma$ and $T'$.
\end{enumerate}
\end{prop}

\begin{proof}
(i) By~Proposition~\ref{pl stuff}(iii), one  can perturb $T$ and $T'$ to intersect minimally.
Then one takes a regular
neighborhood $\Sigma_1$ of $T\cup T'$ and applies Lemma~\ref{create Seifert}
to find a large, \incomp\ Seifert \svar\ $\Sigma_2\subset M$
containing $\Sigma_1$. Then we can enlarge $\Sigma_2$ as follows:

For every pair of tori $T_1,T_2$ in $\bord\Sigma_2$ that are parallel outside $\Sigma_2$, add
a product region, obtaining a graph \svar\ $\Sigma_3$. Then for every boundary torus $T\subset\bord\Sigma_3$, add a region $X$ as in Lemma~\ref{region X} (in Case (ii) take $X=\emptyset$).
We claim that the resulting \svar\ $\Sigma$ is a taut graph \var. The only part of the definition that
is not obvious is that any PL least area \incomp\ torus which can be homotoped into
$\Sigma$ lies in $\Sigma$. Let $T_1$ be such a torus, and let $T'_1\subset\Sigma$ be a
torus homotopic to $T_1$. Since every component of $\bord\Sigma$ is PL least area,
$T_1$ is disjoint from them by~\ref{pl stuff}(iii). Hence if $T_1\not\subset\Sigma$, then $T_1\cap\Sigma=\emptyset$. This implies that $T_1,T'_1$ are disjoint, hence parallel. A product region between them must
contain some component $T_2$ of $\bord \Sigma$. Thus $T_1$ is homotopic to $T_2$, and
must be contained in $\Sigma$.

(ii) We distinguish three cases.

 If $T\subset\Sigma$ we can simply put $\Sigma':=\Sigma$.
 
 If $T$ intersects essentially some
component $T'$ of $\bord\Sigma$, then after perturbation of $T$, $T\cap T'$ is a union of special curves. Let $\Sigma_1$ be a regular neighborhood of $\Sigma\cup T$. Then $\Sigma_1$ is Seifert fibered. The
rest of the proof is as in (i).

If $T\cap\Sigma=\emptyset$, let $T'$ be a least PL area \incomp\ torus intersecting $T$ essentially.
If $T'\cap\Sigma\neq\emptyset$, then using the previous case we can find a taut graph \svar\ $\Sigma''$
containing $\Sigma\cup T'$, and then we are reduced to the first two cases. Otherwise
we let $\Sigma_1$
be the union of $\Sigma$ and a regular neighborhood of $T\cup T'$, and argue as in (i).
\end{proof}

\section{Proof of Theorem~\ref{th:graph}}\label{sec:proof}
\subsection{An increasing sequence of taut graph submanifolds}

\begin{prop}\label{prop:seq}.
There exists a (finite or infinite) sequence $\Sigma_0\subset\Sigma_1\subset \Sigma_2\subset \cdots \Sigma_n \cdots$ of taut graph \svar s of $M$ such that any noncanonical, least PL area \incomp\ torus
is contained in some $\Sigma_n$. Furthermore, each $\Sigma_n$ is the union of
a compact \svar\ with a finite (possibly empty) collection of \svar s homeomorphic to $T^2\times
[0,\infty)$; in particular, $\bord\Sigma_n$ has finitely many connected components.
\end{prop}

\begin{proof}
Let $\mathcal L = [T_1], [T_2], \ldots$ be a list of  all homotopy classes of noncanonical \incomp\ tori.
The construction is by induction using Proposition~\ref{tautening}.

Set $\Sigma_0:=\emptyset$. If the list $\mathcal L$ is nonempty, take a least PL area
representative of the first class $[T_1]$.
Since it is noncanonical, another \incomp\ torus intersects $T_1$ essentially. Proposition~\ref{tautening}(i) gives a \svar\ $\Sigma_1$. Then, assuming we have constructed $\Sigma_1,\ldots,\Sigma_n$, we take a least PL area representative of $[T_{n+1}]$ and apply Proposition~\ref{tautening}(ii). 
\end{proof}

From now on we fix a sequence $\Sigma_n$ satisfying the conclusion of
Proposition~\ref{prop:seq}. If $\Sigma_n$ is finite or eventually constant, then
Theorem~\ref{th:graph} follows immediately by taking $\Sigma:=\bigcup \Sigma_n$.
Henceforth we assume (by going to a subsequence) that $\Sigma_n$ is strictly increasing.

Let $\sigma$ be a $3$-simplex of $\calT$. Then $\sigma\cap\bigcup_n\bord\Sigma_n$  is a disjoint union of normal disks. We say that two such
disks $D,D'$ are \bydef{equivalent} if they are normally homotopic and there exists
$n$ and  a component $X$ of $\Sigma_n\cap\sigma$ that contains both $D$ and $D'$.
This relation is obviously reflexive and symmetric. Since the sequence $\Sigma_n$
is increasing, it is also transitive. Hence it is an equivalence relation.

\begin{lem}\label{finite canonical}
For each $3$-simplex $\sigma$, only finitely many components of $\bigcup_n\bord\Sigma_n$ that
intersect $\sigma$ can belong to canonical tori.
\end{lem}

\begin{proof}
This follows immediately from Hypothesis A and Proposition~\ref{loc fin}.
\end{proof}

\begin{lem}\label{same torus}
For each $3$-simplex $\sigma$, only finitely many components of $\bigcup_n\bord\Sigma_n$ have disconnected intersection with $\sigma$.
\end{lem}

\begin{proof}
Assume the contrary. Let $T_1,\ldots,T_k,\ldots$ be an infinite sequence of tori of $\bigcup_n\bord\Sigma_n$ having disconnected intersection with $\sigma$. Since each $\Sigma_n$ has only finitely many boundary components, we may assume by taking a subsequence that there
is a sequence $n(k)$ going to infinity with $k$, such that for each $k$, $T_k\subset\bord \Sigma_{n(k)}$
but $T_k\not\subset\bord\Sigma_{n(k)-1}$.

Apply hypothesis~C to each $T_k$, getting a torus $T'_k$. Then the $T'_k$'s have uniformly bounded diameter, and their distances from $\sigma$ are uniformly bounded. Hence there are only finitely many of them up to normal homotopy.

It follows that some $T'_{k_0}$ intersects essentially infinitely many of the $T_k$'s. Then
$T'_{k_0}$  is noncanonical, hence contained in some $\Sigma_{n_0}$. This is a contradiction.
\end{proof}

\begin{prop}\label{finite class}
For each $3$-simplex $\sigma$, there are only finitely many equivalence classes of disks
in $\sigma\cap\bigcup_n\bord\Sigma_n$. 
\end{prop}

\begin{proof}
The proof is by contradiction. Suppose that there is an infinite collection $D_k$ of normally homotopic
disks in $\sigma\cap\bigcup_n\bord\Sigma_n$ that are not pairwise equivalent.
By Lemma~\ref{same torus}, one can assume without loss of generality that all those disks,
as well as all disks that are equivalent to them, belong to different tori of $\bigcup_n\bord\Sigma_n$.
Moreover, we may assume that all those tori have connected intersection with $\sigma$,
so that hypothesis~D can be applied to them. We also
assume using Lemma~\ref{finite canonical} that all those tori are noncanonical.

Let $T_1,T_2$ be such that $D_1=T_1\cap \bigcup_n\bord\Sigma_n$ and
$D_2=T_2\cap \bigcup_n\bord\Sigma_n$.

\begin{lem}\label{same seifert}
There exists $n$ and a connected Seifert \svar\ $\Sigma'_n$ of $\Sigma_n$ such that
$T_1\subset \Sigma'_n$ and $T_2\subset\Sigma'_n$.
\end{lem}

\begin{proof}
By hypothesis, $T_1$ and $T_2$ are PL least area and noncanonical, so there exists $n_1,n_2$
such that $T_i\subset \Sigma_{n_i}$ for $i=1,2$. Setting $n:=\max{(n_1,n_2)}$ and using
the fact that the sequence $\Sigma_n$ is increasing, one has $T_i\subset \Sigma_n$.
Let $\Sigma'_1,\Sigma'_2$ be the maximal Seifert components of the graph \var\ $\Sigma_n$
containing $T_1,T_2$ respectively. If $\Sigma'_1=\Sigma'_2$, then the lemma is
proved. Otherwise $\Sigma'_1\cap \Sigma'_2$ is the empty set.

By hypothesis~D, there is an annulus $A$ connecting a special
curve on $T_1$ to a special curve on $T_2$. By a homotopy, we can ensure that $A$ is a union
$A_1\cup_{c_1} A_2\cup_{c_2} \cdots \cup_{c_{k-1}} A_k$, where each $c_j$ is an essential
simple closed curve on $A$, and the $A_j$'s are either contained in $\Sigma'_1\cup
\Sigma'_2$ or have interior disjoint from that set. Furthermore, we may assume that each
$A_j$ that is contained in  $\Sigma'_1\cup \Sigma'_2$ is vertical.

Choose $j$ such that $A_j$ has interior disjoint from $\Sigma'_1\cup \Sigma'_2$ and
connects some component $T'_1$ of $\bord\Sigma'_1$ to some component $T'_2$ of
$\bord\Sigma'_2$. Since the Seifert \var\ $\Sigma'_1$ is large, its base orbifold contains
an essential properly embedded arc $\alpha_1$ connecting the projection of
$T'_1$ to itself. Hence $\Sigma'_1$ contains an essential annulus $A'_1$ connecting
$T'_1$ to itself. Likewise, $\Sigma'_2$ contains an essential annulus $A'_2$ connecting
$T'_2$ to itself. By patching together $A'_1$, $A'_2$, and two parallel copies of $A_j$,
we get a torus $T$ which is vertical in some Seifert \var\ $\Sigma$ containing
both $\Sigma'_1$ and $\Sigma'_2$. By construction, $T$
is \incomp\ and is not homotopic into either $\Sigma'_1$ or
$\Sigma'_2$. Hence it is noncanonical. Still denote by $T$ a least PL area representative
of the homotopy class of $T$. Then there exists $n'>n$ such that $T\subset\Sigma_{n'}$.

It follows that some connected Seifert \svar\ of $\Sigma_{n'}$ contains $\Sigma'_1\cup \Sigma'_2$,
and therefore the original tori $T_1$ and $T_2$ as well.
\end{proof}

Having proved Lemma~\ref{same seifert}, we continue the proof of Proposition~\ref{finite class}.
Let $X\subset\sigma$ be the product region between $D_1$ and $D_2$. If $X\subset\Sigma'$,
then $D_1$ and $D_2$ are equivalent, contradicting our hypothesis. Thus $\bord\Sigma'_n$
intersects $X$ in at least two nonequivalent normal disks $D'_1,D'_2$ 
Furthermore, the components
$T'_1,T'_2$ of $\bord\Sigma'_n$ containing $D'_1,D'_2$ respectively are different.
Let $X'$ denote the product region between $D'_1$ and $D'_2$. Without
loss of generality we assume that $X'\cap \Sigma'_n=\emptyset$.

By hypothesis~D, there is an annulus $A$ connecting a special
curve of $T'_1$ to a special curve of $T'_2$
and containing an arc $\alpha\subset X'$, essential on $A$, and connecting $D'_1$ to
$D'_2$. Since $\Sigma'$ is connected, there is an arc $\alpha'\subset\Sigma'$ connecting
the endpoints of $\alpha$. Thus $\alpha\cup\alpha'$ is a simple closed curve intersecting
each of $T'_1$ and $T'_2$ in a single point. Since $\Sigma'_n$ is Seifert fibered and
$\bord A$ consists of special curves, there is a properly embedded annulus $A'\subset \Sigma'$
that connects both components of $\bord A$, and we may assume that $\alpha'\subset
\Sigma'$.

Let $T$ be the union of $A$ and $A'$. Then $T$ is an embedded, \incomp\ torus, that intersects
$T'_1$ and $T'_2$ essentially. In particular, it is noncanonical. Hence there exists $n'>n$,
a Seifert \svar\ $\Sigma'_{n'}\subset\Sigma_{n'}$ and a torus $T'$ homotopic to $T$
such that $\Sigma'_n\subset\Sigma'_{n'}$ and $T'\subset \Sigma'_{n'}$.

If $\Sigma'_{n'}$ contains $X'$, then $D'_1$ and $D'_2$ are equivalent, contradicting
an assumption made earlier. Hence there is a torus $T''\subset \bord\Sigma'_{n'}$ such
that $T''\cap X'$ is a disk $D''$. Now $\alpha$ intersects $D''$ in an odd number of points,
and $\alpha'\cap T''=\emptyset$ because $\alpha'\subset\Sigma'_n\subset\Sigma'_{n'}$.
Hence $\alpha\cup\alpha'$ is not homotopically disjoint from $T''$. Since $T'$ is
homotopic to $T$, it contains a curve homotopic to $\alpha\cup\alpha'$. This implies
that $T'$ and $T''$ intersect essentially, which contradicts the fact that $T''$ is a boundary
component of some \svar\ containing $T'$.
\end{proof}

Modify $\Sigma_n$ in the following way: for
each edge $e$ of $\cal T$ we look at the set $S_e:= e\cap \bigcup_n\bord\Sigma$. We can
define an equivalence relation on $S_e$ by saying that two points $x,y$ are equivalent if
for some $3$-simplex $\sigma$ containing $e$ there are equivalent disks $D,D'$ in $\sigma\cap\bigcup_n\bord\Sigma_n$ such that $D\cap e=x$ and $D'\cap e=x'$. By
Proposition~\ref{finite class}, this relation also has only finitely many classes.

For every equivalence class $c$, let $I_c$ denote the closure of the convex hull of $c$. Here
we encounter a technical difficulty: perhaps the $I_c$'s are not disjoint, so the closure of
the union of the $\Sigma_n$ is not a \svar, because some points are approached from both sides.

To deal with this, choose arbitrarily a segment $I'_c$ contained in the interior of
$I_c$. Then map linearly $I_c$ onto $I'_c$. Let $c'$ be the image of $c$ under this
mapping. From the union of all those $c$'s, we can, by taking convex hulls in each
$3$-simplex, reconstruct a collection of normal tori bounding graph \svar s. Each of these
new tori is normally homotopic to an old one. We keep the same notation.

Let $\Sigma$ be the closure of the union of the $\Sigma_n$. Then $\Sigma$ is a
\svar\ of $M$, whose interior is the union of the $\Sigma_n$. By construction,
it satisfies properties~(ii) and (iii) of the conclusion of Theorem~\ref{th:graph}.
The remaining task is to show that $\Sigma$ is a graph submanifold of $M$.

\subsection{End of the proof}
In this subsection, we are interested in the boundary components of $\Sigma$. We shall
prove that they are \incomp\ annuli or canonical tori. In order to study them, we need
to make sense of the notion that they are approximated by boundary components of the $\Sigma_n$'s, which are normal tori. To this effect we give the following definition:
\begin{defi}
We say that a sequence $F_n$ of normal surfaces \bydef{converges to} a normal
surface $F$ if the following requirements are fulfilled:

\begin{enumerate}
\item For any compact normal subsurface $K\subset F$ (in the sense of
Section~\ref{sec:pl}), there exists $n_0(K)$ such that for all
$n\ge n_0(K)$, there is a normal subsurface $K_n\subset F_n$ normally homotopic to $K$; 
\item If $K_n\subset F_n$ is any sequence of compact normal subsurfaces having the property that
for all sufficiently large $n,n'$, $K_n$ is normally homotopic to $K_{n'}$, then there
exists a subsurface $K\subset F$ which is normally homotopic to $K_n$ for
all sufficiently large $n$.
\end{enumerate}
\end{defi}

The following lemma is immediate from the construction in the previous subsection:
\begin{lem}\label{approx tori}
Let $F$ be a component of $\bord\Sigma$.
There is a sequence $T_n$ of tori such that for every $n$, $T_n$ is a component of
$\bord\Sigma_n$, and which converges to $F$.
\end{lem}

We next record an important consequence of hypothesis B:
\begin{lem}\label{lem:B}
There is a constant $C'_2$ such that if $T$ is a noncanonical normal least weight torus in $M$,
and  $D$ is a normal subdisk of $T$, then $\diam(D)\le \diam(\bord D) + C'_2$.
\end{lem}

\begin{proof}
Apply hypothesis B to a point $x\in D$ whose distance to $\bord D$ is maximal, noting
that a special curve is noncontractible in $M$, hence cannot lie entirely in $D$.
\end{proof}

\begin{lem}\label{topo F}
The following assertions hold:
\begin{enumerate}
\item $F$ is \incomp\ in $M$;
\item $F$ is a torus or an annulus;
\item If $F$ is a torus, then it is canonical.
\end{enumerate}
\end{lem}

\begin{proof}
(i) We argue by contradiction. Let $D$ be a compressing disk for $F$ whose boundary
is in general position.  By a sequence of istotopies that reduce the length, we may
assume that $\bord D$ is normal in the sense of section~\ref{sec:pl}. Now $\bord D$ is contained in some compact normal subsurface $K$ of $F$. Applying Lemma~\ref{approx tori} and using part~(i) of the definition of convergence
of sequences of normal surfaces, we see that there is for $n$ large enough a normal subsurface $K_n\subset T_n$ normally homotopic to $K$. Hence we can find a normal curve $\gamma_n\subset K_n$ which is connected
to $\bord D$ by a small annulus $A_n$, and we find compression disks $D_n:=A_n\cup D$
for $T_n$. Since each $T_n$ is \incomp, there is for each large $n$ a disk $D'_n\subset T_n$
with $\bord D_n=\bord D'_n$.

By construction the diameter of $\bord D'_n$ is independent of $n$. Hence Lemma~\ref{lem:B}
gives a uniform upper bound for the diameters of the $D'_n$'s.
By Lemma~\ref{myineq}, the weights of the $D'_n$ are also uniformly bounded above. Hence there
are only finitely many of them up to normal homotopy. Using part~(ii) of the definition
of convergence of sequences of normal surfaces, we get a disk $D'\subset F$ such that
$\bord D=\bord D'$.

(ii) If $F$ is compact, then Lemma~\ref{approx tori} implies that
$F$ is a torus. Hence we suppose that $F$ is noncompact.

Let $\gamma\subset F$ be a simple closed curve. We say that it is \bydef{special}
if it is normally homotopic to a special curve $\gamma_n\subset T_n$ for large $n$.
In particular, any special curve in $F$ is essential in $M$, hence essential in $F$.
\begin{slem}\label{unif rep}
There is a constant $C''_2$ such that for every $x\in F$, there is a special curve $\gamma\subset F$ of length at most $C''_2$ whose distance to $x$ is at most $C''_2$.
\end{slem}

\begin{proof}
Take $x\in F$. Then $x$
is a limit of a sequence $x_n\in T_n$. By hypothesis~B of the main theorem, through
each $x_n$ there is a special curve
$\gamma_n\subset T_n$ of length at most $C_2$. For each $n$ we perform a normalizing
sequence of homotopies, getting a normal special curve $\gamma'_n\subset T_n$.
The bound on the length of $\gamma_n$ gives  bounds on both the length of $\gamma'_n$
and the distance between $\gamma'_n$ and $x_n$. Hence there are only finitely many
$\gamma'_n$'s up to normal isotopy. This shows that a subsequence of $\gamma'_n$
converges to some normal special curve $\gamma\subset F$ whose length and distance
from $x$ can be bounded above by a constant $C''_2$ depending only on $C_2$.
This proves Sublemma~\ref{unif rep}.
\end{proof}

Next we show that  $F$ is planar. By way of contradiction, suppose that $F$ is nonplanar.
Then it contains a nonplanar compact normal subsurface $K$. Let $K_n \subset T_n$ be a sequence
of approximating subsurfaces. Then for each $n$,
$K_n$ has genus $1$.  Hence its boundary consists of curves of uniformly
bounded length that bound disks on $T_n$. By Lemma~\ref{lem:B},
those disks have unifomly bounded diameter. This implies that the $T_n$'s have uniformly bounded diameter, contradicting the
noncompactness of the limit $F$. This contradiction proves that $F$ is planar.

By Sublemma~\ref{unif rep}, $F$ contains an essential curve, so it cannot be homeomorphic
to $\Rr^2$. If it had more than two ends, then we could find a compact
subsurface $K\subset F$
such that $F\setminus K$ has at least three noncompact components $U_1,U_2,U_3$.
For $i=1,2$ pick a point $x_i\in U_i$ such that $d(x_i,K)>C_2$. By Sublemma~\ref{unif rep},
there are special curves $\gamma_i\subset U_i$ for $i=1,2$. Now special curves are
homotopic in $F$, as can be seen by approximation in some $T_n$. This contradiction
proves that $F$ has two ends. Hence it is an annulus and (ii) is proven.

(iii) If $F$ is a torus, then for $n$ sufficiently large, $\bord\Sigma$ contains a torus
$T_n$ normally homotopic to $F$. By uniqueness of the least PL area representative
of a normal homotopy class, the $T_n$'s are equal. If they were not canonical,
then there would exist a least PL area $T'$ that is not homotopically disjoint from $T_n$.
For large $n$, $T'$ would have to be contained in $\Sigma_n$.  This is a contradiction.
\end{proof}
 
At last we show that $\Sigma$ is a graph submanifold of $M$: let $Z$ be a union of PL least area canonical tori
and Klein bottles such that every canonical torus contained in $\Sigma$ is homotopic to some component
of $Z$, or a double cover of some component of $Z$. By hypothesis~A
and Proposition~\ref{loc fin}, $Z$ is a \svar. Let $U$ be a regular neighborhood of $Z$
in $\Sigma$. All we have to show is that every component of $\Sigma\setminus U$
is Seifert fibered.

Let $X$ be such a component. Choose for each annular component $A_i$ of $\bord X$ an annulus $A'_i\subset X$ properly
homotopic to $A_i$, in such a way that the $A'_i$'s do not intersect one
another, and that for each $n$, $A'_i\cap \Sigma_n$
is empty or an annulus whose core is a special curve. Then the $A'_i$'s together
with the toral components of $\bord X$ bound a \svar\ $X'\subset X$ such that
$X$ retraction deforms onto $X'$. In particular, $X$ and $X'$ are homeomorphic, so
it is enough to prove that $X'$ is Seifert fibered.

In order to do this, we notice that the \svar s $\Sigma_n\cap X'$ can be given compatible
Seifert fibrations such that for every annulus $A'_i$ and every $n$, if $A'_i\cap \Sigma_n$
is nonempty, then it is a vertical annulus. This gives a Seifert fibration on
$X\setminus \bigcup_i A_i$ such that every $A'_i$ is vertical. This Seifert fibration
restricts to a Seifert fibration on $X'$. Hence $\Sigma$ is a graph submanifold of $M$,
and the proof of Theorem~\ref{th:graph} is complete.

\bibliographystyle{abbrv}
\bibliography{jsj}

Institut de Recherche Math\'ematique Avanc\'ee, Universit\'e Louis
Pasteur, 7 rue Ren\'e Descartes, 67084 Strasbourg Cedex, France\\ 
\texttt{maillot@math.u-strasbg.fr}

\end{document}